\documentclass[11pt, a4paper]{article}
\usepackage{HongLiu}
\usepackage{xcolor}

\hypersetup{
  colorlinks=true,
  linkcolor=blue,
  citecolor=blue,
  urlcolor=blue
}

\tikzstyle{aNode} = [circle, fill = black]
\tikzstyle{bNode} = [circle,draw = black, thick]

\theoremstyle{plain}
\newtheorem{corollary}[theorem]{Corollary}
\newtheorem{proposition}[theorem]{Proposition}
\newtheorem{property}[theorem]{Property}

\theoremstyle{remark}
\newtheorem{remark}[theorem]{Remark}

\newcommand\pc{{\pi_{\operatorname{col}}}}
\newcommand\pcd{{\pi_{\operatorname{col}}^\delta}}

\newcommand{\dens}{\rho}
\newcommand{\opp}{\mathrm{op}}
\newcommand{\Fcal}{\mathcal F}
\newcommand{\Dcal}{\mathcal D}
\newcommand{\Hcal}{\mathcal H}
\newcommand{\Picol}{\Pi_{\mathrm{col}}^{\mathrm{fin}}}
\newcommand{\Ldig}{\Lambda_{\mathrm{dig}}}

\DeclareTextCompositeCommand{\v}{OT1}{l}{l\nobreak\hspace{-.1em}'}
\DeclareTextCompositeCommand{\v}{OT1}{t}{t\nobreak\hspace{-.1em}'\nobreak\hspace{-.15em}}

\title{Colored Tur\'an densities and finite directed templates}
\author{
Nannan ~Chen\thanks{School of Mathematics, Shandong University, Jinan, China. Email: \texttt{chennannan@mail.sdu.edu.cn}.}
\and
Ander ~Lamaison\thanks{Universidad P\'ublica de Navarra, Pamplona, Spain. Email: \texttt{ander.lamaison@unavarra.es}} 
\and
Heng~Li\thanks{School of Mathematics, Shandong University, Jinan, China, and Extremal Combinatorics and Probability Group (ECOPRO), Institute for Basic Science (IBS), Daejeon, South
Korea. Supported by the National Natural Science Foundation of China (12501487), by China Scholarship
Council and the Institute for Basic Science (IBS-R029-C4). Email: \texttt{heng.li@sdu.edu.cn}.}
\and
Hong~Liu\thanks{Extremal Combinatorics and Probability Group (ECOPRO), Institute for Basic Science (IBS), Daejeon, South
Korea. Supported by the Institute for Basic Science (IBS-R029-C4). Email: \texttt{hongliu@ibs.re.kr}.}}
\date{}

\begin{document}
\maketitle

\begin{abstract}
Classical Tur\'an theory relates the asymptotic extremal problem for an ordinary forbidden graph to its chromatic number.  This principle fails when a forbidden copy must respect prescribed vertex colors, as occurs naturally in multipartite extremal problems and in embeddings with prescribed clusters.  We study properly vertex-colored graphs whose hosts have a prescribed density between every pair of color classes, with containment interpreted through color-preserving homomorphisms.  Our main theorem gives a finite directed reduction: for every fixed $m$ and $\varepsilon>0$, every sufficiently many-colored $d$-dense host admits a colored homomorphism from an $m$-layer half-blowup $D[m]$ of a finite digraph $D$ with density at least $d-\varepsilon$, where $D$ may be chosen with $\lceil\varepsilon^{-1}\rceil$ vertices.  Consequently, the colored Tur\'an density of every non-empty finite forbidden family equals the homomorphism Tur\'an density of its associated acyclic orientations.  We then determine the complete finite value set: these densities are exactly the Lagrangians of finite loopless digraphs, and hence are rational.  The main technical input is a low-resolution sampling method that preserves all pairwise bipartite degree profiles on a large subfamily of color classes while keeping the sample size independent of the number of retained classes.
\end{abstract}

\section{Introduction}

%How much edge density forces a graph with a prescribed vertex-color pattern?  Without the color requirement, 
How much edge density forces a graph to contain a prescribed subgraph? Without any additional requirement, the answer is governed by a single familiar invariant.  For a graph $F$ with at least one edge, let $\operatorname{ex}(n,F)$ be the maximum number of edges in an $n$-vertex $F$-free graph.  Tur\'an's theorem~\cite{Tur} and the Erd\H{o}s--Stone--Simonovits theorem~\cite{ErdosSimonovits66,ErdS46} give
\[
  \pi(F):=\lim_{n\to\infty}\frac{\operatorname{ex}(n,F)}{\binom n2}
  =1-\frac{1}{\chi(F)-1}.
\]
Thus the asymptotic extremal density depends only on the chromatic number: in particular, it vanishes for every bipartite graph.

%Prescribed vertex classes lead to a different question.  
In graphs with a given partition, vertex class requirements lead to a different question. A forbidden configuration may require several vertices to lie in the same class and others to lie in distinct classes.  Such requirements arise in multipartite extremal problems and when an embedding must respect a given partition, as in regularity-based arguments.  The relevant density condition is then imposed between every pair of classes, rather than on the graph as a whole.  We work in the multipartite setting introduced by Bollob\'as~\cite{Bol78}, expressed in terms of proper vertex colorings.

A \emph{$k$-colored graph} is a graph with a proper $k$-coloring in which every color class is non-empty.  It is \emph{$d$-dense} if the bipartite graph between every two distinct color classes has edge density at least $d$.
A \emph{colored homomorphism} from a colored graph $F$ to a colored graph $G$ is a graph homomorphism $\varphi:V(F)\to V(G)$ such that $\varphi(u)$ and $\varphi(v)$ have the same color if and only if $u$ and $v$ have the same color.  An injective colored homomorphism identifies $F$ with a \emph{colored subgraph} of $G$.

Thus color names may be relabeled, but distinct source colors must use distinct host classes.  Vertices of the same color may be identified.  This homomorphic formulation allows arbitrary non-empty host classes; the injective formulation has the same density threshold when all host classes are sufficiently large, by the usual supersaturation argument with an arbitrarily small density margin.

Let $\Fcal$ be a non-empty finite family of colored graphs, and let $k$ be at least the maximum number of colors in a member of $\Fcal$.  Define $\pc(\Fcal,k)$ to be the infimum of the values of $d$ such that every $d$-dense $k$-colored graph admits a colored homomorphism from some $F\in\Fcal$.  The \emph{colored Tur\'an density} of $\Fcal$ is
\[
\pc(\Fcal)=\lim_{k\to\infty}\pc(\Fcal,k).
\]
The limit exists because $\pc(\Fcal,k)$ is non-increasing in $k$.  For a singleton family, we write $\pc(F)$.

Equivalently, $\pc(\Fcal,k)$ is the infimum of the values of $d$ for which some $n$ guarantees a colored copy of a member of $\Fcal$ in every $d$-dense $k$-colored graph whose classes all have size at least $n$.  Throughout, forbidden families are non-empty whenever a density parameter is discussed.  A family containing an edgeless colored graph has density zero.

When $F$ is rainbow, meaning that all its vertices have distinct colors, this is the multipartite Tur\'an problem studied by Bondy, Shen, Thomass\'e and Thomassen~\cite{BonSTT06}.  They proved, among other results, that $\pc(K_3,3)=\frac{\sqrt5-1}{2}$ and that $\pc(F)=\pi(F)$ for every rainbow $F$.  Pfender~\cite{Pfe12} showed that the limiting value for cliques is attained for all sufficiently large $k$, and Narins and Tran~\cite{NarT17} characterized the graphs $F$ achieving $\pc(F,k)=\pi(F)$ for all large enough $k$.  The equality $\pc(F)=\pi(F)$ also holds when the prescribed coloring uses exactly $\chi(F)$ colors.  Indeed, sampling one vertex from each host class gives a transversal of the required average density, and the ordinary Tur\'an theorem forces a rainbow $K_{\chi(F)}$ onto which $F$ maps.  The matching lower bound comes from an ordinary $(\chi(F)-1)$-partite construction.

General color patterns behave differently, even on paths.  The \emph{$k$-half-graph} has vertices $a_1,\dots,a_k,b_1,\dots,b_k$, with $a_i,b_i$ of color $i$, and edges $a_i b_j$ for $i<j$.  It is $\frac14$-dense.  The path $P_5$ with color sequence $1-2-3-1-2$ admits no colored homomorphism into any half-graph, and hence has positive colored Tur\'an density despite being bipartite.  We will show later that its density is exactly $\frac14$.  The obstruction comes from the order of the classes: $a_i b_j$ is an edge in one direction of the order, but not in the other.  Our first main result shows that this directed structure governs the general problem.

\begin{figure}[ht]
\centering
\begin{tikzpicture}[x=0.75pt,y=0.75pt,yscale=-1.2,xscale=1.2]
%uncomment if require: \path (0,235); %set diagram left start at 0, and has height of 235

%Rounded Rect [id:dp18931150533789765] 
\draw   (140,57.91) .. controls (140,54.65) and (142.65,52) .. (145.91,52) -- (163.64,52) .. controls (166.91,52) and (169.56,54.65) .. (169.56,57.91) -- (169.56,115.2) .. controls (169.56,118.46) and (166.91,121.11) .. (163.64,121.11) -- (145.91,121.11) .. controls (142.65,121.11) and (140,118.46) .. (140,115.2) -- cycle ;
%Shape: Circle [id:dp9116556237641933] 
\draw  [fill={rgb, 255:red, 0; green, 0; blue, 0 }  ,fill opacity=1 ] (153.33,66.5) .. controls (153.33,65.67) and (154,65) .. (154.83,65) .. controls (155.66,65) and (156.33,65.67) .. (156.33,66.5) .. controls (156.33,67.33) and (155.66,68) .. (154.83,68) .. controls (154,68) and (153.33,67.33) .. (153.33,66.5) -- cycle ;
%Shape: Circle [id:dp9177858533103352] 
\draw  [fill={rgb, 255:red, 0; green, 0; blue, 0 }  ,fill opacity=1 ] (153.33,105.83) .. controls (153.33,105) and (154,104.33) .. (154.83,104.33) .. controls (155.66,104.33) and (156.33,105) .. (156.33,105.83) .. controls (156.33,106.66) and (155.66,107.33) .. (154.83,107.33) .. controls (154,107.33) and (153.33,106.66) .. (153.33,105.83) -- cycle ;
%Rounded Rect [id:dp06115191558208122] 
\draw   (200,57.91) .. controls (200,54.65) and (202.65,52) .. (205.91,52) -- (223.64,52) .. controls (226.91,52) and (229.56,54.65) .. (229.56,57.91) -- (229.56,115.2) .. controls (229.56,118.46) and (226.91,121.11) .. (223.64,121.11) -- (205.91,121.11) .. controls (202.65,121.11) and (200,118.46) .. (200,115.2) -- cycle ;
%Shape: Circle [id:dp007696978323782355] 
\draw  [fill={rgb, 255:red, 0; green, 0; blue, 0 }  ,fill opacity=1 ] (213.33,66.5) .. controls (213.33,65.67) and (214,65) .. (214.83,65) .. controls (215.66,65) and (216.33,65.67) .. (216.33,66.5) .. controls (216.33,67.33) and (215.66,68) .. (214.83,68) .. controls (214,68) and (213.33,67.33) .. (213.33,66.5) -- cycle ;
%Shape: Circle [id:dp21043261726007834] 
\draw  [fill={rgb, 255:red, 0; green, 0; blue, 0 }  ,fill opacity=1 ] (213.33,105.83) .. controls (213.33,105) and (214,104.33) .. (214.83,104.33) .. controls (215.66,104.33) and (216.33,105) .. (216.33,105.83) .. controls (216.33,106.66) and (215.66,107.33) .. (214.83,107.33) .. controls (214,107.33) and (213.33,106.66) .. (213.33,105.83) -- cycle ;
%Straight Lines [id:da2203854516169762] 
\draw    (154.83,66.5) -- (214.33,105.83) ;
%Rounded Rect [id:dp20206108863755612] 
\draw   (260,57.91) .. controls (260,54.65) and (262.65,52) .. (265.91,52) -- (283.64,52) .. controls (286.91,52) and (289.56,54.65) .. (289.56,57.91) -- (289.56,115.2) .. controls (289.56,118.46) and (286.91,121.11) .. (283.64,121.11) -- (265.91,121.11) .. controls (262.65,121.11) and (260,118.46) .. (260,115.2) -- cycle ;
%Shape: Circle [id:dp20448707493924356] 
\draw  [fill={rgb, 255:red, 0; green, 0; blue, 0 }  ,fill opacity=1 ] (273.33,66.5) .. controls (273.33,65.67) and (274,65) .. (274.83,65) .. controls (275.66,65) and (276.33,65.67) .. (276.33,66.5) .. controls (276.33,67.33) and (275.66,68) .. (274.83,68) .. controls (274,68) and (273.33,67.33) .. (273.33,66.5) -- cycle ;
%Shape: Circle [id:dp40968388703557157] 
\draw  [fill={rgb, 255:red, 0; green, 0; blue, 0 }  ,fill opacity=1 ] (273.33,105.83) .. controls (273.33,105) and (274,104.33) .. (274.83,104.33) .. controls (275.66,104.33) and (276.33,105) .. (276.33,105.83) .. controls (276.33,106.66) and (275.66,107.33) .. (274.83,107.33) .. controls (274,107.33) and (273.33,106.66) .. (273.33,105.83) -- cycle ;
%Rounded Rect [id:dp2534445619278062] 
\draw   (320,57.91) .. controls (320,54.65) and (322.65,52) .. (325.91,52) -- (343.64,52) .. controls (346.91,52) and (349.56,54.65) .. (349.56,57.91) -- (349.56,115.2) .. controls (349.56,118.46) and (346.91,121.11) .. (343.64,121.11) -- (325.91,121.11) .. controls (322.65,121.11) and (320,118.46) .. (320,115.2) -- cycle ;
%Shape: Circle [id:dp5477335212205747] 
\draw  [fill={rgb, 255:red, 0; green, 0; blue, 0 }  ,fill opacity=1 ] (333.33,66.5) .. controls (333.33,65.67) and (334,65) .. (334.83,65) .. controls (335.66,65) and (336.33,65.67) .. (336.33,66.5) .. controls (336.33,67.33) and (335.66,68) .. (334.83,68) .. controls (334,68) and (333.33,67.33) .. (333.33,66.5) -- cycle ;
%Shape: Circle [id:dp1323614313801228] 
\draw  [fill={rgb, 255:red, 0; green, 0; blue, 0 }  ,fill opacity=1 ] (333.33,105.83) .. controls (333.33,105) and (334,104.33) .. (334.83,104.33) .. controls (335.66,104.33) and (336.33,105) .. (336.33,105.83) .. controls (336.33,106.66) and (335.66,107.33) .. (334.83,107.33) .. controls (334,107.33) and (333.33,106.66) .. (333.33,105.83) -- cycle ;
%Rounded Rect [id:dp7269905027905943] 
\draw   (380,57.91) .. controls (380,54.65) and (382.65,52) .. (385.91,52) -- (403.64,52) .. controls (406.91,52) and (409.56,54.65) .. (409.56,57.91) -- (409.56,115.2) .. controls (409.56,118.46) and (406.91,121.11) .. (403.64,121.11) -- (385.91,121.11) .. controls (382.65,121.11) and (380,118.46) .. (380,115.2) -- cycle ;
%Shape: Circle [id:dp7841073355568655] 
\draw  [fill={rgb, 255:red, 0; green, 0; blue, 0 }  ,fill opacity=1 ] (393.33,66.5) .. controls (393.33,65.67) and (394,65) .. (394.83,65) .. controls (395.66,65) and (396.33,65.67) .. (396.33,66.5) .. controls (396.33,67.33) and (395.66,68) .. (394.83,68) .. controls (394,68) and (393.33,67.33) .. (393.33,66.5) -- cycle ;
%Shape: Circle [id:dp06481754874552936] 
\draw  [fill={rgb, 255:red, 0; green, 0; blue, 0 }  ,fill opacity=1 ] (393.33,105.83) .. controls (393.33,105) and (394,104.33) .. (394.83,104.33) .. controls (395.66,104.33) and (396.33,105) .. (396.33,105.83) .. controls (396.33,106.66) and (395.66,107.33) .. (394.83,107.33) .. controls (394,107.33) and (393.33,106.66) .. (393.33,105.83) -- cycle ;
%Straight Lines [id:da6620091434408107] 
\draw    (214.83,66.5) -- (274.33,105.33) ;
%Straight Lines [id:da5964883134430061] 
\draw    (275.33,67) -- (334.83,105.83) ;
%Straight Lines [id:da18212675028009984] 
\draw    (334.83,66.5) -- (394.33,105.33) ;
%Straight Lines [id:da676118702035365] 
\draw    (154.83,66.5) -- (274.33,105.33) ;
%Straight Lines [id:da5976370733812586] 
\draw    (154.83,66.5) -- (334.83,105.83) ;
%Straight Lines [id:da7415437867243518] 
\draw    (154.58,66.25) -- (393.33,105.83) ;
%Straight Lines [id:da2698707603340673] 
\draw    (214.83,66.5) -- (334.83,105.83) ;
%Straight Lines [id:da23280771236397702] 
\draw    (275.33,67) -- (395.33,106.33) ;
%Straight Lines [id:da3181439633108908] 
\draw    (213.96,66.37) -- (394.83,105.83) ;

% Text Node
\draw (118,60) node [anchor=north west][inner sep=0.75pt]   [align=left] {$a_1$};
% Text Node
\draw (118,99) node [anchor=north west][inner sep=0.75pt]   [align=left] {$b_1$};
% Text Node
\draw (415,60) node [anchor=north west][inner sep=0.75pt]   [align=left] {$a_5$};
% Text Node
\draw (415,99) node [anchor=north west][inner sep=0.75pt]   [align=left] {$b_5$};
% Text Node
\draw (146,127) node [anchor=north west][inner sep=0.5pt]   [align=left] {$V_1$};
% Text Node
\draw (206,127) node [anchor=north west][inner sep=0.5pt]   [align=left] {$V_2$};
% Text Node
\draw (266,127) node [anchor=north west][inner sep=0.5pt]   [align=left] {$V_3$};
% Text Node
\draw (327,127) node [anchor=north west][inner sep=0.5pt]   [align=left] {$V_4$};
% Text Node
\draw (388,127) node [anchor=north west][inner sep=0.5pt]   [align=left] {$V_5$};

\end{tikzpicture}
\caption{The $5$-half-graph, with $a_i$ and $b_i$ in the same color class.}
\end{figure}
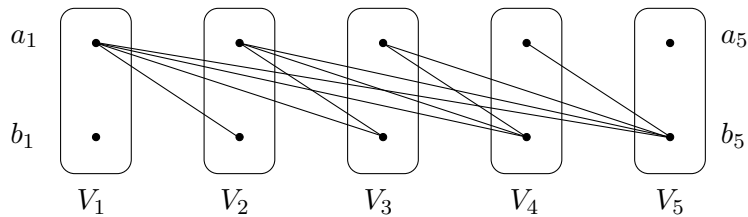

\subsection{Finite directed templates}

Let $D$ be a finite digraph, possibly with loops and pairs of opposite arcs.  Its \emph{$m$-half-blowup}, denoted by $D[m]$, is the $m$-colored graph with color classes $V_1,\dots,V_m$, each a copy of $V(D)$.  For $i<j$, the vertices $u_i\in V_i$ and $v_j\in V_j$ are adjacent if and only if $uv$ is an arc of $D$.

We normalize digraph density by $\dens(D)=\frac{e(D)}{|V(D)|^2}$, counting a loop as one arc.  Every pair of classes in $D[m]$ has density $\dens(D)$.  A single non-loop arc gives the half-graph, while two vertices joined in both directions give a complete bipartite graph with the same-color matching removed; see Figure~\ref{fig:blowup}.

\begin{figure}[ht]
\centering
\includegraphics[width=0.7\textwidth]{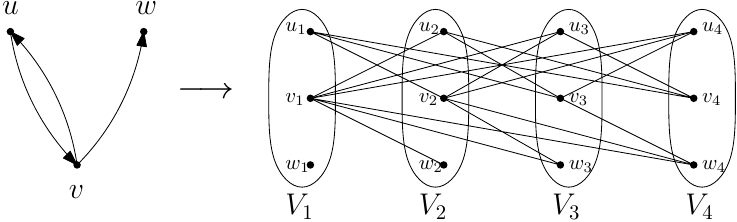}
\caption{A digraph $D$ and its $4$-half-blowup $D[4]$.}
\label{fig:blowup}
\end{figure}

In a half-blowup, every ordered pair of classes carries the same adjacency matrix.  An arbitrary colored graph need not have any such consistency.  Nevertheless, our first result states that, after an arbitrarily small density loss, it always contains a homomorphic image of a large half-blowup.  Crucially, the number of vertices in the directed template is independent of the number of classes to be realized.

\begin{theorem}\label{thm:blowup}
For every $\varepsilon>0$ and every positive integer $m$, there exists $k=k(\varepsilon,m)\ge m$ such that every $d$-dense $k$-colored graph $G$ admits a colored homomorphism from $D[m]$ for some possibly looped digraph $D$ satisfying
$  \dens(D)\ge d-\varepsilon$
and $  |V(D)|=\lceil\varepsilon^{-1}\rceil.$
\end{theorem}

The bound $\lceil\varepsilon^{-1}\rceil$ is optimal as a bound uniform in $m$; see Remark~\ref{rem:sharp-template}.  Thus, for limiting density questions, half-blowups are not merely a useful source of examples: they suffice to witness every obstruction.  To make this precise, we translate the prescribed colors into orientations.

For a colored graph $F$ and a linear order $\preceq$ of its colors, let $F^{\preceq}$ be the digraph obtained by orienting each edge from the smaller color to the larger color.  Write $\Dcal(F)$ for the family of all such orientations, and put $\Dcal(\Fcal)=\bigcup_{F\in\Fcal}\Dcal(F)$.

\begin{figure}[ht]
\centering
\includegraphics[width=0.8\textwidth]{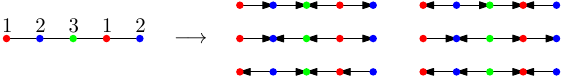}
\caption{The path with color sequence $1-2-3-1-2$ and its associated orientations.  The top-left orientation corresponds to $1\prec2\prec3$.}
\end{figure}

Every member of $\Dcal(F)$ is acyclic.  If $F$ has at most $m$ colors, then $F$ maps coloredly into $D[m]$ exactly when some $F^{\preceq}$ maps into $D$: the order of the image classes specifies $\preceq$, and projection onto $D$ gives the directed homomorphism.  For a digraph family $\Hcal$, let $\pi(\Hcal)$ be the supremum of the densities of non-empty finite digraphs receiving no homomorphism from any member of $\Hcal$, with value zero if no such host exists.  Theorem~\ref{thm:blowup} now gives an exact reduction.

\begin{theorem}\label{thm:digraph}
For every non-empty finite family $\Fcal$ of colored graphs,
\[
  \pc(\Fcal)=\pi(\Dcal(\Fcal)).
\]
\end{theorem}

This identifies the invariant replacing chromatic number: the homomorphism obstructions formed by all color-order orientations.  It also connects the problem to the directed extremal theory of Brown, Erd\H{o}s and Simonovits~\cite{BroES73,BES1985}, where weighted finite digraphs and quadratic optimization describe extremal constructions.

\subsection{The value set and its consequences}

For a non-empty finite digraph $D$, its \emph{Lagrangian} is
\[
  \lambda(D)=\max_{\substack{x_v\ge0\\ \sum_{v\in V(D)}x_v=1}}
  \sum_{uv\in E(D)}x_u x_v,
\]
where a loop $vv$ contributes $x_v^2$. The weights represent the relative sizes of vertex classes in a blowup of $D$, so $\lambda(D)$ is the largest density approached by its finite vertex-blowups.  Our second main theorem determines exactly which of these values arise from finite colored forbidden families.  Set
\begin{align*}
  \Picol&=\{\pc(\Fcal):\Fcal\text{ is a non-empty finite family of colored graphs}\},\\
  \Ldig&=\{\lambda(D):D\text{ is non-empty, finite and loopless}\}.
\end{align*}

\begin{theorem}\label{thm:value-set}
$\Picol=\Ldig$.
\end{theorem}

The two inclusions have different content.  For a given forbidden family, acyclicity bounds the support of an optimal directed template: a semicomplete digraph with too many vertices contains a large transitive tournament, and hence one of the forbidden orientations.  Conversely, every finite loopless digraph can be realized by a finite colored forbidden family.  We prove this directly by excluding the finitely many bounded templates with larger Lagrangian.  Thus, the value-set theorem requires neither an infinite approximation nor a well-ordering argument.

\begin{corollary}\label{cor:rational}
For every non-empty finite colored graph family $\Fcal$, the number $\pc(\Fcal)$ is rational.
\end{corollary}

If every forbidden member has an edge, the density is attained by the half-blowups of a finite loopless digraph avoiding the prescribed family.  This extends the familiar finite extremal constructions from ordinary graph Tur\'an theory.  The complete bidirected digraph on $r-1$ vertices has Lagrangian $1-\frac{1}{r-1}$ and recovers the ordinary chromatic-number bound; single arcs and genuinely directed templates produce further values.

The Brown--Erd\H{o}s--Simonovits well-ordering theorem~\cite{BES1985} also transfers to the colored setting.  We record the resulting gap property separately; it is not needed for Theorem~\ref{thm:value-set}.

\begin{property}\label{property:no-right}
For every $\alpha\in[0,1)$ there exists $\varepsilon>0$ such that
$ (\alpha,\alpha+\varepsilon)\cap\Picol=\varnothing$.
\end{property}

For a single forbidden graph, the first gap has a particularly simple explanation.  Every digraph with at least one arc receives a homomorphism from the one-arc digraph, whose half-blowups are the half-graphs.

\begin{corollary}\label{cor:jump}
A $k$-colored graph $F$ satisfies $\pc(F)=0$ if and only if it admits a colored homomorphism into the $k$-half-graph.  In particular, $\pc(F)\notin(0,\frac14)$.  The value $\frac14$ is attained by $P_5$ with color sequence $1-2-3-1-2$.
\end{corollary}

The chromatic number no longer determines the answer, but it still gives sharp bounds on its possible range.

\begin{corollary}\label{cor:chromatic}
Let $F$ be a colored graph with $\chi(F)\ge2$.  Then
\[
  1-\frac{1}{\chi(F)-1}
  \le \pc(F)
  <1-\frac{1}{2(\chi(F)-1)}.
\]
For every integer $\ell\ge2$, the upper endpoint $1-1/(2(\ell-1))$ is an accumulation point of colored Tur\'an densities of graphs with chromatic number $\ell$.
\end{corollary}

Here $\chi(F)$ is the chromatic number of the underlying graph, not the number of prescribed colors.  In particular, $\pc(F_i)\to1$ holds exactly when $\chi(F_i)\to\infty$.

For ordinary graphs, every finite forbidden family is \emph{principal}: its Tur\'an density is the minimum of the Tur\'an densities of its members.  Non-principality is a familiar distinction of hypergraph Tur\'an theory~\cite{Balogh02,MubayiPikhurko08}.  It already occurs for graphs with prescribed vertex colors.

\begin{corollary}\label{cor:nonprincipal}
There exist colored graphs $F_1,F_2$ such that
  $\pc(F_1,F_2)<\min\{\pc(F_1),\pc(F_2)\}.$
\end{corollary}

Our proof gives an explicit example: $F_1=K_2\times K_{17}$, colored by its second coordinate, and a rainbow triangle $F_2$.  Their joint density is $\frac13$, whereas $\pc(F_1)\ge\frac38$ and $\pc(F_2)=\frac12$.

\subsection{Low-resolution sparsification}

The main technical ingredient is a sampling theorem for a prescribed partition.  It does not assume that the graph avoids any configuration, or even that its pair densities are large.  Given multisets $A,B$ in distinct color classes, let $d(A,B)$ be the proportion of adjacent pairs in $A\times B$, counting occurrences separately.

\begin{lemma}[Low-resolution sparsification]\label{lem:lowres}
For every $\varepsilon>0$ there exists $s=s(\varepsilon)$ such that, for every positive integer $m$, there exists $k=k(\varepsilon,m)$ with the following property.  Every $k$-colored graph with classes $V_1,\dots,V_k$ has a set $I\subseteq[k]$ of size $m$ and multisets $S_i$ of $s$ vertices from $V_i$, for $i\in I$, such that
\[
  |d(S_i,S_j)-d(V_i,V_j)|\le\varepsilon
  \qquad\text{for all distinct }i,j\in I.
\]
For $0<\varepsilon\le1$, one may take $s=O(\varepsilon^{-3}\log(2/\varepsilon))$.
\end{lemma}

The important feature is the order of the parameters: the sample size depends only on the error, not on the number of retained classes.  Independent sampling and a union bound preserve all pairs of $m$ specified classes with $O(\varepsilon^{-2}\log m)$ vertices per class.  Lemma~\ref{lem:lowres} removes the dependence on $m$ by allowing the classes themselves to be selected from a larger family.  Once the samples have bounded size, only finitely many bipartite adjacency matrices remain, and Ramsey's theorem finds $m$ classes with a common matrix.  This is the directed template in Theorem~\ref{thm:blowup}.

The proof adapts the low-resolution method used by Lamaison~\cite{Lam24} to obtain palette descriptions of uniform hypergraph Tur\'an density.  Its basic idea is to sample in several rounds and control interactions between different rounds.  When a new block is selected, the earlier blocks are fixed, so the necessary concentration estimates involve only the new sample.  Ramsey's theorem then selects classes on which these estimates hold simultaneously.  Interactions within one round are left uncontrolled, but contribute only a small fraction of the final sample.

We prove a stronger form, Theorem~\ref{thm:technical}, preserving both degree distributions of every selected bipartite pair and the relative degrees of each sampled vertex.  It therefore handles density and minimum-degree assumptions in the same argument.  We expect this sparsification theorem to be useful independently whenever pairwise data on a prescribed partition must be retained with bounded samples.  Its conclusion is different from regularity: it controls every pair in a selected subfamily of the original classes, rather than almost every pair of a newly constructed partition.  It does not assert control of arbitrary configurations spanning three or more classes.

\subsection{Minimum degree}

The stronger sparsification theorem also yields a local version of the directed reduction.  In this setting density and minimum degree behave differently: the gap above zero disappears.

A $k$-colored graph with classes $V_1,\dots,V_k$ has \emph{degree at least $\delta$} if every vertex outside $V_i$ has at least $\delta|V_i|$ neighbors in $V_i$, for every $i$. If both the minimum in-degree and the minimum out-degree of $D$ are at least $\delta|V(D)|$, then $D[k]$ has degree at least $\delta$.

\begin{theorem}\label{thm:blowdeg}
For every $\varepsilon>0$ there exists $s=s(\varepsilon)$ such that, for every positive integer $m$, there exists $k=k(\varepsilon,m)\ge m$ with the following property.  Every $k$-colored graph $G$ of degree at least $\delta$ admits a colored homomorphism from $D[m]$ for some possibly looped $s$-vertex digraph $D$ whose minimum in-degree and minimum out-degree are both at least $(\delta-\varepsilon)s$.
\end{theorem}

For a non-empty finite colored graph family $\Fcal$, let $\pcd(\Fcal)$ be the supremum of those $\delta\in[0,1]$ for which, for every $k$, there is a $k$-colored graph of degree at least $\delta$ receiving no colored homomorphism from any member of $\Fcal$.
For a digraph family $\Hcal$, let $\pi^\delta(\Hcal)$ be the supremum of the normalized minimum semidegrees
$\min\{\delta^+(D),\delta^-(D)\}/|V(D)|$
over non-empty finite digraphs $D$ receiving no homomorphism from any member of $\Hcal$.  In both definitions an empty feasible set has value zero, and singleton families are abbreviated as before.

\begin{theorem}\label{thm:degdigraph}
For every non-empty finite colored graph family $\Fcal$,
$\pcd(\Fcal)=\pi^\delta(\Dcal(\Fcal))$.
\end{theorem}

\begin{theorem}\label{thm:nojump}
For every $\varepsilon>0$ there exists a colored graph $F$ such that
$0<\pcd(F)<\varepsilon$.
\end{theorem}

The contrast with Corollary~\ref{cor:jump} reflects a basic directed distinction.  Positive edge density exposes a single arc; positive minimum semidegree forces a directed cycle, but its length need not stay bounded as the semidegree tends to zero.  Periodically colored cycles make this distinction explicit, connecting the minimum-degree problem to directed girth and the Caccetta--H\"aggkvist conjecture.

The paper is organized as follows.  Section~\ref{sec:tools} collects the preliminary tools.  Section~\ref{sec:method} states the stronger sparsification theorem and the finite-support and obstruction lemmas, whose proofs appear in Section~\ref{sec:method-proofs}.  Section~\ref{sec:main-proofs} proves the density results, including the direct realization theorem.  Section~\ref{sec:min-degree-proofs} treats minimum degree, and Section~\ref{sec:remarks} discusses the remaining questions.

\section{Preliminaries and tools}\label{sec:tools}

We collect the probabilistic, homomorphism and Lagrangian facts used below.  The well-ordering theorem is needed only for Property~\ref{property:no-right}.
We need the Chernoff--Hoeffding inequality.
\begin{lemma}[\cite{Hoe63}]\label{lem:hoeffding}
Let $X_1,X_2,\dots,X_t$ be independent Bernoulli random variables, each taking the value $1$ with probability $p$. Then
  $\Pr\left(\left|\frac{1}{t}\sum_{i=1}^t X_i-p\right|>\varepsilon\right)
  \le 2e^{-2\varepsilon^2t}.$
\end{lemma}

We will also use Ramsey's theorem. It states that for all integers $k,n,r$ there exists $N=R_r(n,k)$ such that, whenever the edges of the complete $r$-uniform hypergraph on $N$ vertices are colored in $k$ colors, there exists a subset of $n$ vertices in which all edges have the same color. If the subindex $r$ is not present, we assume $r=2$.

\subsection{Half-blowup and directed homomorphism}

The next lemma relates colored homomorphisms into half-blowups to directed homomorphisms.

\begin{lemma}\label{lem:half-blowup-test}
Let $F$ be a colored graph with $r$ colors, let $D$ be a possibly looped digraph, and let $m\ge r$.  Then $F$ admits a colored homomorphism into $D[m]$ if and only if there is a linear order $\preceq$ of the colors of $F$ such that $F^{\preceq}$ admits a digraph homomorphism to $D$.
\end{lemma}

\begin{proof}
Suppose first that $\psi:F\to D[m]$ is a colored homomorphism.  Distinct colors of $F$ are mapped to distinct color classes of $D[m]$, so the image layers induce an order $\preceq$ on the colors of $F$.  Let $p:D[m]\to D$ be the projection onto the underlying copy of $D$.  If $uv$ is an arc of $F^{\preceq}$, then the color of $u$ is mapped to an earlier layer than the color of $v$.  Since $\psi(u)\psi(v)$ is an edge of $D[m]$, the definition of half-blowup gives the arc $p(\psi(u))p(\psi(v))$ in $D$.  Thus $p\circ\psi$ is a digraph homomorphism $F^{\preceq}\to D$.

Conversely, suppose $f:F^{\preceq}\to D$ is a digraph homomorphism.  Choose an order-preserving injection from the color set of $F$, ordered by $\preceq$, into $[m]$.  Send a vertex $v$ of color $c$ to $(f(v),i_c)$, where $i_c$ is the layer assigned to $c$.  If $uv$ is an edge of $F$ and the color of $u$ is smaller than the color of $v$, then $uv$ is an arc of $F^{\preceq}$, so $f(u)f(v)$ is an arc of $D$.  Therefore $(f(u),i_{\chi(u)})$ is adjacent to $(f(v),i_{\chi(v)})$ in $D[m]$.  This gives a colored homomorphism $F\to D[m]$.
\end{proof}

\subsection{Lagrangians and the finite value set}

We record the basic properties of the Lagrangian. For a finite possibly looped digraph $D$, its \emph{density} is defined by \(\dens(D)=\frac{|E(D)|}{|V(D)|^2}, \) where loops are counted as arcs. The \emph{opposite digraph} $D^{\opp}$ is obtained by reversing every arc; loops remain loops. Note that for a non-empty finite loopless digraph, $\lambda(D)<1$.

\begin{lemma}\label{lem:lag-basic}
Let $A$ and $B$ be non-empty finite possibly looped digraphs.  Then the following hold.
\begin{enumerate}[label=\textup{(\roman*)}]
  \item $\dens(A)\le\lambda(A)$.
  \item If $A\to B$, then $\lambda(A)\le\lambda(B)$.
  \item $\lambda(A^{\opp})=\lambda(A)$.
\end{enumerate}
\end{lemma}

\begin{proof}
For \textup{(i)}, use uniform weights.  For \textup{(ii)}, let $\phi:A\to B$ and, for a weight vector $x$ on $A$, put $y_b=\sum_{a\in\phi^{-1}(b)}x_a$.  Then $\sum_{uv\in E(A)}x_u x_v
  \le \sum_{bc\in E(B)}y_b y_c
  \le\lambda(B).$
Taking the maximum proves the claim.  Reversing arcs does not change the polynomial, giving \textup{(iii)}.
\end{proof}

\begin{lemma}\label{lem:lag-rational}
For every non-empty finite loopless digraph $D$, the value $\lambda(D)$ is rational.  Moreover, $\lambda(D)$ is the exact density of a finite loopless vertex-blowup of a subdigraph of $D$.
\end{lemma}

\begin{proof}
Write $P_D(x)=\frac12 x^{\mathsf T}Mx$, where $M$ is the symmetric integer matrix with zero diagonal and
$m_{ij}=\mathbf1_{ij\in E(D)}+\mathbf1_{ji\in E(D)}$ for $i\ne j$.
Choose an optimal vector $x$, and let $S$ be its support.  Stationarity on this face of the simplex gives
\[
  M[S]x=\theta\mathbf1,
  \qquad \mathbf1^{\mathsf T}x=1,
  \qquad x\ge0,
\]
where $\theta=2\lambda(D)$.  These equations and inequalities define a non-empty rational polytope in the variables $(x,\theta)$, so it has a rational point $(y,\theta')$.  Symmetry gives
\[
  \theta=y^{\mathsf T}M[S]x
  =x^{\mathsf T}M[S]y=\theta'.
\]
Hence $P_D(y)=\frac{\theta'}{2}=\lambda(D)$, giving a rational optimal vector and a rational value.

Delete the zero coordinates of $y$ and write $y_i=\frac{n_i}{N}$ with positive integers $n_i$ summing to $N$.  Replacing vertex $i$ by $n_i$ independent twins produces a loopless digraph of density
$N^{-2}\sum_{ij\in E(D[\operatorname{supp}(y)])}n_i n_j=\lambda(D)$.
Its natural projection to $D$ is a homomorphism.
\end{proof} 

\subsection{Brown--Erd\H{o}s--Simonovits well-ordering}

For a finite loopless digraph family $\Hcal$, the \emph{subdigraph-forbidding density} $\pi_{\mathrm{sub}}(\Hcal)$ is the usual Tur\'an density for forbidding members of $\Hcal$ as subdigraphs, normalized by $\frac{e(G)}{|V(G)|^2}$.

A subset of $\mathbb R$ is \emph{well-ordered} if each of its non-empty subsets has a least element.

\begin{theorem}[Brown--Erd\H{o}s--Simonovits~\cite{BES1985}]\label{thm:BES}
The set of densities
$\{\pi_{\mathrm{sub}}(\Hcal)\}$, where $\Hcal$ is a finite loopless digraph family, is well-ordered under the usual order of the real numbers.
\end{theorem}

\begin{lemma}\label{lem:hom-sub}
Let $\Hcal$ be a finite family of loopless digraphs.  Then there exists a finite family $\Hcal'$ of loopless digraphs such that $\pi(\Hcal)=\pi_{\mathrm{sub}}(\Hcal')$.
\end{lemma}

\begin{proof}
If $\Hcal=\varnothing$, take $\Hcal'=\varnothing$.  If $\Hcal$ contains an edgeless non-empty digraph, then every non-empty host admits a homomorphism from that member, and the density is $0$; this is obtained by taking $\Hcal'$ to be the one-arc digraph.  We may therefore assume that every member of $\Hcal$ has at least one arc.

A host with a loop admits a homomorphism from every such member by mapping all vertices to the looped vertex.  Hence homomorphism-free hosts may be assumed loopless.  Let $\operatorname{Im}_0(\Hcal)$ be the finite family of all loopless homomorphic images of members of $\Hcal$.  For every loopless host $G$, there is a homomorphism $H\to G$ for some $H\in\Hcal$ if and only if $G$ contains some member of $\operatorname{Im}_0(\Hcal)$ as a subdigraph. Thus the finite $\Hcal$-homomorphism-free digraphs are precisely the $\operatorname{Im}_0(\Hcal)$-free digraphs. Since equal blow-ups preserve freeness and asymptotically preserve density, the supremum over finite $\operatorname{Im}_0(\Hcal)$-free digraphs equals $\pi_{\mathrm{sub}}(\operatorname{Im}_0(\Hcal))$. Therefore $\pi(\Hcal)=\pi_{\mathrm{sub}}(\operatorname{Im}_0(\Hcal))$.
\end{proof}

\section{Methods}\label{sec:method}

\subsection{Low-resolution sparsification and degree profiles}

Lemma~\ref{lem:lowres} preserves all pair densities on a large subfamily of a prescribed partition.  We will prove a stronger statement that also preserves normalized degree distributions and the relative degrees of individual sampled vertices.  This additional information is what allows the same method to treat minimum degree.  For a vertex $v$ and a non-empty multiset $B$, write $d(v,B)$ for the fraction of occurrences in $B$ adjacent to $v$.

\begin{corollary}\label{cor:main}
For every $\varepsilon>0$ there exists $s=s(\varepsilon)$ such that, for every positive integer $m$, there exists $k=k(\varepsilon,m)$ with the following property.  Every $d$-dense $k$-colored graph contains $m$ color classes with $s$-element multisets $S_1,\dots,S_m$ satisfying $d(S_i,S_j)\ge d-\varepsilon$ for all $i\ne j$.
\end{corollary}

This is immediate from Lemma~\ref{lem:lowres}.  To pass from Corollary~\ref{cor:main} to Theorem~\ref{thm:blowup}, label the $s$ occurrences in each sample.  There are at most $2^{s^2}$ bipartite adjacency matrices, so Ramsey's theorem finds a large subfamily on which every ordered pair has the same matrix.  A final sampling argument reduces the directed template to $\lceil\varepsilon^{-1}\rceil$ vertices.  The independence of $s$ from the number of retained classes is essential: it allows the Ramsey number to be chosen after $s$ has been fixed.

For a bipartite graph on $(A,B)$, its \emph{degree profile on $A$} is the tail function
\[
  f_{A,B}(\mu)=\frac{|\{a\in A:d(a,B)\ge\mu\}|}{|A|},
  \qquad 0\le\mu\le1.
\]
Occurrences are counted separately when a class is a multiset.  Two bipartite graphs on $(A,B)$ and $(A',B')$ are \emph{$\varepsilon$-profile-close} if
\[
  \int_0^1|f_{A,B}(\mu)-f_{A',B'}(\mu)|\,d\mu\le\varepsilon,
\]
and the analogous inequality holds after interchanging the two sides.  Since the integral of a degree profile is the edge density, profile closeness implies that the two edge densities differ by at most $\varepsilon$.

For multisets $S$ of vertices of $A$ and $T$ of vertices of $B$, we say that $(S,T)$ is an \emph{$\varepsilon$-discretization} of $G[A,B]$ if $G[S,T]$ and $G[A,B]$ are $\varepsilon$-profile-close and, in addition,
\[
  |d(v,T)-d(v,B)|\le\varepsilon\quad(v\in S),
  \qquad
  |d(w,S)-d(w,A)|\le\varepsilon\quad(w\in T).
\]
Thus the profile condition controls the degree distribution, while the two pointwise conditions prevent an exceptional sampled vertex from losing many of its neighbors.

\begin{theorem}[Low-resolution sparsification with degree control]\label{thm:technical}
For every $\varepsilon>0$ there exist $s=s(\varepsilon)$ and $n=n(\varepsilon)$ such that, for every positive integer $m$, there exists $k=k(\varepsilon,m)$ with the following property.  Every $k$-colored graph with classes $V_1,\dots,V_k$ of size at least $n$ has $m$ classes $V_{a_1},\dots,V_{a_m}$ and subsets $S_i\subseteq V_{a_i}$ of size $s$ such that every pair $(S_i,S_j)$, $i\ne j$, is an $\varepsilon$-discretization of $G[V_{a_i},V_{a_j}]$.

For $0<\varepsilon\le1$, one may take $s=O(\varepsilon^{-3}\log(2/\varepsilon))$.  In particular, pair density at least $d$ is preserved up to $\varepsilon$, and degree at least $\delta$ is preserved as minimum degree at least $(\delta-\varepsilon)s$ between every two selected subsets.
\end{theorem}

\begin{proof}[Proof of Lemma~\ref{lem:lowres} from Theorem~\ref{thm:technical}]
Replace each vertex in $V_i$ by $t_i$ independent twins, where $t_i$ is chosen so that $t_i|V_i|\ge n(\varepsilon)$.  This preserves every pair density.  Apply Theorem~\ref{thm:technical} to the enlarged graph and project the selected subsets back to the original classes.  The projections are multisets of the same size, and their pair densities are unchanged.  Profile closeness therefore gives the required density approximation.  The bound on $s$ is inherited from Theorem~\ref{thm:technical}.
\end{proof}

\subsection{Finite supports and realization}

The directed reduction leaves two questions: why an extremal value has a finite Lagrangian witness, and how to realize an arbitrary finite digraph Lagrangian by colored obstructions.  The first follows from a support reduction.  The second uses a simple way to make an obstruction survive every ordering of its color classes.

A digraph is \emph{semicomplete} if every pair of distinct vertices is joined by at least one arc.  Let $T_h$ denote the transitive tournament on $h$ vertices.

\begin{proposition}\label{prop:acyclic-lagrangian}
Let $\Hcal$ be a non-empty finite family of loopless acyclic digraphs.  Then $\pi(\Hcal)=\lambda(Q)$ for some non-empty finite loopless digraph $Q$.
If every member of $\Hcal$ has an arc and $h=\min\{|V(H)|:H\in\Hcal\}$, then $Q$ may be chosen semicomplete, receiving no homomorphism from any member of $\Hcal$, and satisfying $|V(Q)|<2^{h-1}$.
\end{proposition}

Indeed, an optimal weight vector of minimal support cannot give positive weight to two nonadjacent vertices: moving their combined weight to one of them does not lower the objective.  The support is therefore semicomplete.  A sufficiently large semicomplete digraph contains a transitive tournament into which an acyclic obstruction maps, bounding the support size.

For realization, the following lemma removes the dependence on the order of the colors.  Recall that $H^{\opp}$ is obtained by reversing every arc of $H$.

\begin{lemma}\label{lem:layered}
Let $D_0$ be a non-empty finite loopless digraph and let $H$ be a finite possibly looped digraph.  Suppose that $H\nrightarrow D_0$ and $H^{\opp}\nrightarrow D_0$.  Set
$ m_H=|V(D_0)|^{|V(H)|}+1$.
Then $(H[m_H])^{\preceq}\nrightarrow D_0$ for every linear order $\preceq$ of the color classes of $H[m_H]$.
\end{lemma}

A homomorphism would give the same map to $D_0$ on two layers, by pigeonhole.  Comparing the order of those two layers recovers a homomorphism from either $H$ or $H^{\opp}$.  In Section~\ref{sec:main-proofs}, we apply this lemma to the finitely many bounded semicomplete templates with Lagrangian larger than $\lambda(D_0)$.  Adding a rainbow clique obstruction bounds the template size and makes this finite exclusion exact.

\section{Proofs of the technical ingredients}\label{sec:method-proofs}

\subsection{Proof of the low-resolution theorem}

\begin{figure}[ht]
\centering
\includegraphics[width=0.8\textwidth]{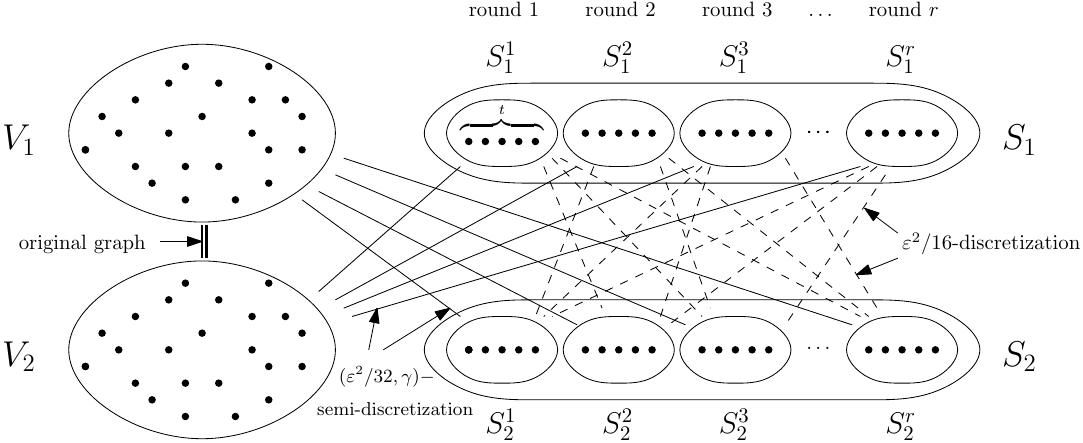}
\caption{Scheme for the vertex selection within the color classes $V_1$ and $V_2$, and the discretization requirements between each pair of sets. Notably, there is no requirement between pairs of sets selected on the same round.}
\end{figure}

We select $t$ vertices per active class in each of $r$ rounds.  Let $S_i^p$ be the block selected from $V_i$ in round $p$.  The invariant controls $(S_i^p,S_j^q)$ only when $p\ne q$.  The uncontrolled same-round pairs contribute a fraction $\frac{1}{r}$ of each final bipartite graph, and each sampled vertex has only one uncontrolled opposite block.  This is enough for both profile and pointwise degree preservation.

The intermediate notion keeps one side unsampled.  A multiset $S$ of vertices of $A$ gives an \emph{$(\varepsilon,\gamma)$-semi-discretization} $(S,B)$ of $G[A,B]$ if $G[S,B]$ is $\varepsilon$-profile-close to $G[A,B]$ and all but $\gamma|B|$ vertices $b\in B$ satisfy $|d(S,b)-d(A,b)|\le\varepsilon$.  As before, occurrences in a multiset are counted separately.

The following lemmas will be the main steps of the proof.

\begin{lemma}\label{lem:semi-disc}
Let $r$ be a positive integer.  Let $G$ be a bipartite graph on $(A,B)$, and let $S=(v_1,\dots,v_t)$ be sampled independently and uniformly with replacement from $A$, viewed as a multiset.  Let $0<\delta\le1$, $0<\gamma\le\frac{\delta}{2}$, and put $z=\lceil\frac{2}{\delta}\rceil$.  If
  $2e^{-\frac{\delta^2t}{2}}\bigl(z+1+\gamma^{-1}\bigr)
  \le \frac{1}{16r},$
then $(S,B)$ is a $(\delta,\gamma)$-semi-discretization of $G$ with probability at least $1-\frac{1}{16r}$.
\end{lemma}

\begin{lemma}\label{lem:full-disc}
Let $r$ be a positive integer.  Let $G$ be a bipartite graph on $(A,B)$, let $(S,B)$ be a $(\frac{\delta}{2},\gamma)$-semi-discretization of $G$ with $|S|=t$, and sample $T=(w_1,\dots,w_t)$ independently and uniformly with replacement from $B$, viewed as a multiset.  Let $0<\delta\le1$, $0<\gamma\le\frac{\delta}{4}$, and put $z'=\lceil\frac{4}{\delta}\rceil$.  If \( \gamma t\le\frac{1}{16r}\) and
$  2e^{-\frac{\delta^2t}{8}}\bigl(z'+1+\gamma^{-1}\bigr)
  \le\frac{1}{16r}$
,
then $(S,T)$ is a $\delta$-discretization of $G$ with probability at least $1-\frac{1}{8r}$.
\end{lemma}

\begin{lemma}\label{lem:disc-comb}
Let $r$ be a positive integer and $0<\varepsilon\le1$.  Let $G$ be a bipartite graph on $(A,B)$, and let $S_1,\dots,S_r$ be multisets of vertices of $A$, and let $T_1,\dots,T_r$ be multisets of vertices of $B$, all of a common positive size $t$. Suppose that $(S_i,T_j)$ is an $\eta$-discretization of $G$ whenever $i\ne j$, where $\eta=\frac{\varepsilon}{4}$.  If $S$ and $T$ are the multiset unions of the $S_i$ and $T_i$, respectively, and $r\ge\frac{4}{\varepsilon}$, then $(S,T)$ is an $\varepsilon$-discretization of $G$.
\end{lemma}

The first two lemmas are concentration statements.  The third combines the rounds by averaging degrees and degree profiles.  We prove them after deriving Theorem~\ref{thm:technical}.

\begin{proof}[Proof of Theorem~\ref{thm:technical}]
We may assume $0<\varepsilon\le1$ and $m\ge2$, since decreasing $\varepsilon$ or increasing $m$ only strengthens the conclusion.  Put
\(
  \eta=\frac{\varepsilon}{4}
\)
 and \(r=\left\lceil\frac{4}{\varepsilon}\right\rceil.
\)
Choose $t=\lceil C\varepsilon^{-2}\log(2/\varepsilon)\rceil$, where $C$ is a sufficiently large absolute constant, and set $\gamma=e^{-\eta^2t/16}$.  Then
$  \gamma\le\frac{\eta}{4}$,
  $\gamma t\le\frac{1}{16r}$,
  and 
  $2e^{-\eta^2t/8}\bigl(\lceil4/\eta\rceil+1+\gamma^{-1}\bigr)
  \le\frac{1}{16r}$.
These are precisely the common conditions needed for Lemma~\ref{lem:semi-disc} with $\delta=\eta/2$ and Lemma~\ref{lem:full-disc} with $\delta=\eta$.  Set
\(
  s=rt,
  \) and 
 \( n=16(r+1)t^2.
\)
Define integers $u_r=m$ and, recursively,
$  u_{q}=R_{u_{q+1}}(4u_{q+1}+1,u_{q+1})$, $0\le q<r$,
and take $k=u_0$.

We construct nested active index sets
\(
  I_0\supseteq I_1\supseteq\cdots\supseteq I_r
\) with \(|I_q|=u_q,
\)
and for each $i\in I_q$, disjoint $t$-sets $S_i^1,\dots,S_i^q\subseteq V_i$.  For $q\ge1$, the invariants after round $q$ are:
\begin{enumerate}[label=\textup{(A\arabic*)}]
  \item for $1\le p\le q$ and $i\ne j$ in $I_q$, $(S_i^p,V_j)$ is an $(\frac{\eta}{2},\gamma)$-semi-discretization of $G[V_i,V_j]$;
  \item for $1\le p<q$ and $i\ne j$ in $I_q$, $(S_i^p,S_j^q)$ is an $\eta$-discretization of $G[V_i,V_j]$.
\end{enumerate}
Once the construction reaches $I_r$, invariant \textup{(A2)}, with the two sides interchanged when necessary, supplies the required $\eta$-discretization for every pair of blocks selected in different rounds.  Lemma~\ref{lem:disc-comb} applied to
\(
  S_i=\bigcup_{q=1}^r S_i^q
\)
then shows that every pair $(S_i,S_j)$ with $i\ne j$ in $I_r$ is an $\varepsilon$-discretization.

Suppose that rounds $1,\dots,q$ have been completed, where $0\le q<r$.  For distinct $i,j\in I_q$, call a random ordered sample $S=(v_1,\dots,v_t)$ from $V_i$ \emph{good with respect to $j$} if it has no repeated vertex, is disjoint from $S_i^1\cup\cdots\cup S_i^q$, the pair $(S,V_j)$ is an $(\frac{\eta}{2},\gamma)$-semi-discretization, and $(S,S_j^p)$ is an $\eta$-discretization for every $1\le p\le q$.

The semi-discretization condition fails with probability at most $\frac{1}{16r}$.  For each $p\le q$, Lemma~\ref{lem:full-disc}, applied to the already constructed pair $(S_j^p,V_i)$ with the sides interchanged, shows that the corresponding full-discretization condition fails with probability at most $\frac{1}{8r}$.  Finally, since at most $qt\le rt$ vertices of $V_i$ have already been used, the probability that the ordered sample repeats a vertex or meets a previously selected vertex is at most
  $\frac{tqt+\binom t2}{|V_i|}
  \le\frac{(r+1)t^2}{n}
  \le\frac{1}{16}.$
A union bound therefore shows that $S$ is good with respect to any fixed $j$ with probability greater than $\frac{3}{4}$.

Set $u=u_{q+1}$.  If some $u$-set $U\subseteq I_q$ has the property that, for every $i\in U$, there is a sample $S_i^{q+1}\subseteq V_i$ good with respect to every $j\in U\setminus\{i\}$, then take $I_{q+1}=U$ and proceed.  Otherwise, for each $u$-set $U=\{\ell_1<\cdots<\ell_u\}$ choose the least position $B(U)\in[u]$ for which no such sample exists for $\ell_{B(U)}$.  This gives a $u$-coloring of the $u$-subsets of $I_q$.  By the definition of $u_q$, there are
\(
  h_1<\cdots<h_{4u+1}
\)
and $\alpha\in[u]$ such that $B(U)=\alpha$ for every $u$-subset of $\{h_1,\dots,h_{4u+1}\}$.

Sample an ordered $t$-tuple $S=(v_1,\dots,v_t)$ independently and uniformly with replacement from $V_{h_{2u+1}}$. Since $S$ is good with respect to each of the other $4u$ indices with probability greater than $\frac{3}{4}$, some realization is good with respect to at least $3u$ of them. A good realization has distinct entries, so we identify it with its underlying $t$-set. At least $u$ good indices lie on each side of $h_{2u+1}$.  Choose $\alpha-1$ from the left and $u-\alpha$ from the right.  Together with $h_{2u+1}$ they form a $u$-set $U$ in which the middle index occupies position $\alpha$, yet the chosen $S$ is good with respect to every other member of $U$.  This contradicts $B(U)=\alpha$.

Thus the construction continues through all $r$ rounds.  The sets $S_i$ for $i\in I_r$ are genuine subsets of size $s$ because the blocks selected within each class are disjoint, and the preceding application of Lemma~\ref{lem:disc-comb} completes the proof.
\end{proof}

We now prove the auxiliary lemmas, beginning with an elementary fact about tail functions.

\begin{lemma}\label{lem:functions}
    Let $f,g:[0,1]\rightarrow[0,1]$ be two non-increasing functions. Let $\tau\in(0,1)$. If $\int_0^1|f(x)-g(x)|dx\geq 2\tau,$ then there exists $x\in[0,1-\tau]$ such that $f(x+\tau)> g(x)+\tau$ or $g(x+\tau)> f(x)+\tau$.
\end{lemma}

\begin{proof}
    Let $F(x)=\max\{f(x),g(x)\}$ and $G(x)=\min\{f(x),g(x)\}$. It suffices to find $x$ such that $F(x+\tau)>G(x)+\tau$. We have \[2\tau\leq\int_0^1(F(x)-G(x))dx=\int_0^\tau (F(x)-G(1-x))dx+\int_0^{1-\tau}(F(x+\tau)-G(x))dx.\]

    On the right-hand side, the term inside the first integral is at most 1, so the integral is at most $\tau$. The second integral therefore evaluates to at least $\tau$, and so there exists $x$ such that $F(x+\tau)-G(x)\geq\frac{\tau}{1-\tau}>\tau$.
\end{proof}

\begin{proof}[Proof of Lemma~\ref{lem:semi-disc}]
For $\mu\in[0,1]$, let
\[
  f(\mu)=\frac{|\{a\in A:d(a,B)\ge\mu\}|}{|A|},
  \qquad
  g(\mu)=\frac{|\{a\in S:d(a,B)\ge\mu\}|}{|S|},
\]
where occurrences in the multiset $S$ are counted separately.  For each grid point $\frac{i}{z}$, the quantity $g(\frac{i}{z})$ is the average of $t$ independent Bernoulli random variables with mean $f(\frac{i}{z})$. Hence, the Chernoff--Hoeffding inequality gives
  $\Pr\left(\left|f(\frac{i}{z})-g(\frac{i}{z})\right|>\frac{\delta}{2}\right)
  \le 2e^{-\frac{\delta^2t}{2}}.$
If the two non-increasing tail functions $f(\mu)$ and $g(\mu)$  have $L^1$-distance greater than $\delta$, by Lemma~\ref{lem:functions} with $\tau=\frac{\delta}{2}$, there exists $x\in[0,1-\tau]$ such that, say, $f(x+\tau)>g(x)+\tau$. Taking $i=\lceil zx\rceil$ we have $x\le \frac{i}{z}\le x+\tau$ as $\frac{1}{z}\le\tau$. Monotonicity then gives
$f(\frac{i}{z})\ge f(x+\tau)>g(x)+\tau\ge g(\frac{i}{z})+\tau$,
so $|f(\frac{i}{z})-g(\frac{i}{z})|>\frac{\delta}{2}$ at a grid point. The other case is symmetric. Hence the first profile condition fails with probability at most $(z+1)2e^{-\frac{\delta^2t}{2}}$. That is,  $\Pr\left(\int_0^1\left|f(\mu)-g(\mu)\right|d\mu > \delta\right) \le (z+1)2e^{-\frac{\delta^2t}{2}}.$

Let $X\subseteq B$ be the set of vertices in $B$ such that $|d(b,S)-d(b,A)|>\frac{\delta}{2}$. Choose a vertex $b\in B$. Applying the Chernoff--Hoeffding inequality to the indicators of adjacency to $b$ gives $\Pr(b\in X)\leq 2e^{-\frac{\delta^2t}{2}}$. By Markov's inequality, we have 
$  \Pr(|X|>\gamma|B|)
  \le \frac{2e^{-\frac{\delta^2t}{2}}}{\gamma}.$
On the complementary event, all but $\gamma|B|$ vertices have the required degree approximation.  Moreover, for any two indexed families $(a_x)$ and $(b_x)$ in $[0,1]$,
\[
  \int_0^1\left|\frac{1}{N}\sum_x\mathbf1_{a_x\ge\mu}
  -\frac{1}{N}\sum_x\mathbf1_{b_x\ge\mu}\right|\,d\mu
  \le \frac{1}{N}\sum_x|a_x-b_x|.
\]
Applying this with $a_b=d(b,A)$ and $b_b=d(b,S)$ shows that the second profile distance is at most $\frac{\delta}{2}+\gamma\le\delta$.  The stated numerical assumption makes the total failure probability at most $\frac{1}{16r}$.
\end{proof}

\begin{proof}[Proof of Lemma~\ref{lem:full-disc}]
Apply Lemma~\ref{lem:semi-disc} with parameter $\frac{\delta}{2}$ to the bipartite graph $G[B,S]$, treating the occurrences of $S$ as distinct vertices.  With probability at least $1-\frac{1}{16r}$, the pair $(T,S)$ is a $(\frac{\delta}{2},\gamma)$-semi-discretization of $G[B,S]$.  Consequently, $G[S,T]$ is $\frac{\delta}{2}$-profile-close to $G[S,B]$, which is itself $\frac{\delta}{2}$-profile-close to $G[A,B]$.  The triangle inequality gives the required $\delta$-profile closeness.

The same event leaves at most $\gamma|S|=\gamma t<1$ occurrences $v\in S$ with
\(
  |d(v,T)-d(v,B)|>\frac{\delta}{2}.
\)
Thus there are no such occurrences.  In the original semi-discretization $(S,B)$, at most $\gamma|B|$ vertices $b\in B$ satisfy
\(
  |d(b,S)-d(b,A)|>\frac{\delta}{2}.
\)
The probability that the random multiset $T$ meets this exceptional set is at most $\gamma t\le\frac{1}{16r}$.  Outside these two failure events, all pointwise degree conditions and both profile conditions hold, so $(S,T)$ is a $\delta$-discretization.  The total failure probability is at most $\frac{1}{8r}$.
\end{proof}

\begin{proof}[Proof of Lemma~\ref{lem:disc-comb}]
For an occurrence $v\in S_i$, all blocks $T_j$ except $T_i$ approximate its degree into $B$ to within $\eta$.  Averaging gives
$  |d(v,T)-d(v,B)|\le\eta+\frac1r\le\frac{\varepsilon}{2}.$
The symmetric bound holds for occurrences in $T$. To compare profiles, use the elementary coupling inequality
\[
  \int_0^1\left|\frac1N\sum_x\mathbf1_{a_x\ge\mu}
  -\frac1N\sum_x\mathbf1_{b_x\ge\mu}\right|\,d\mu
  \le\frac1N\sum_x|a_x-b_x|
\]
for indexed families in $[0,1]$.  Fix $i$ and choose $j\ne i$.  The profile of $S_i$ into $T_j$ is within $\eta$ of that of $A$ into $B$.  Moreover, for every $v\in S_i$,
$|d(v,T_j)-d(v,T)|\le2\eta+\frac1r$.
Thus the profile of $S_i$ into $T$ is within $3\eta+\frac1r\le\varepsilon$ of the original profile.  Since all $S_i$ have the same size, averaging their tail functions gives the same bound for the profile of $S$ into $T$.  Interchanging the two sides proves the other profile bound.
\end{proof}

\subsection{Finite-support and obstruction lemmas}

\begin{proof}[Proof of Proposition~\ref{prop:acyclic-lagrangian}]
If $\Hcal$ contains an edgeless digraph, then $\pi(\Hcal)=0$, realized by the Lagrangian of the one-vertex loopless digraph.  Otherwise, every $\Hcal$-homomorphism-free host is loopless, since a loop receives a constant homomorphism from every member of $\Hcal$.  Let $\mathscr G(\Hcal)$ be the class of non-empty finite loopless digraphs receiving no such homomorphism.  Lemmas~\ref{lem:lag-basic} and~\ref{lem:lag-rational} give
\begin{equation}\label{eq:hom-density-sup-lag}
  \pi(\Hcal)=\sup_{G\in\mathscr G(\Hcal)}\lambda(G).
\end{equation}
Indeed, $\dens(G)\le\lambda(G)$, while $\lambda(G)$ is the density of a finite blowup projecting to $G$ and hence still belonging to $\mathscr G(\Hcal)$.

For $G\in\mathscr G(\Hcal)$, choose an optimal weight vector $x$ with minimal support $S$.  If $u,v\in S$ are nonadjacent in both directions, the objective is affine in $x_u$ when $x_u+x_v$ and all other weights are fixed.  Moving their combined weight to one endpoint cannot decrease it, contradicting minimality of $S$.  Thus $Q=G[S]$ is semicomplete and $\lambda(Q)=\lambda(G)$.

Choose $H_0\in\Hcal$ with $|V(H_0)|=h$.  A topological ordering embeds $H_0$ into $T_h$, so every member of $\mathscr G(\Hcal)$ is $T_h$-free.  Every semicomplete digraph on at least $2^{h-1}$ vertices contains $T_h$: choose one direction on each pair and use the standard induction giving a transitive $h$-vertex subtournament in every tournament of that order.  Hence $|V(Q)|<2^{h-1}$.

The supremum in \eqref{eq:hom-density-sup-lag} is therefore a maximum over finitely many semicomplete $\Hcal$-homomorphism-free digraphs.  Any maximizing member has the required properties.
\end{proof}

\begin{proof}[Proof of Lemma~\ref{lem:layered}]
Suppose that $\varphi:(H[m_H])^{\preceq}\to D_0$ is a homomorphism.  Each layer $i$ gives a map $p_i:V(H)\to V(D_0)$, defined by $p_i(v)=\varphi(v,i)$.  There are fewer possible maps than layers, so $p_i=p_j=:p$ for some $i<j$.

For every arc $uv\in E(H)$, the half-blowup contains the edge $(u,i)(v,j)$.  If layer $i$ precedes layer $j$ under $\preceq$, this edge forces $p(u)p(v)\in E(D_0)$, giving $H\to D_0$.  If their order is reversed, it forces $p(v)p(u)\in E(D_0)$, giving $H^{\opp}\to D_0$.  Both contradict the hypothesis.
\end{proof}

\section{Proof of the main density results and applications}\label{sec:main-proofs}

\subsection{The half-blowup theorem and the directed reduction}

We first use low-resolution sparsification and Ramsey's theorem to obtain a directed template of bounded size.  Sampling this template once more gives the optimal order $\lceil\varepsilon^{-1}\rceil$; integrality of the number of arcs absorbs the small density loss in the first step.

\begin{proof}[Proof of Theorem~\ref{thm:blowup}]
For $\varepsilon\ge1$, use the one-vertex edgeless digraph.  Assume $0<\varepsilon<1$, and set
  $t=\lceil\varepsilon^{-1}\rceil$,
$\gamma=\frac{1}{2t^2(t-1)}$ and 
$m_0=\max\{m,2\}.$
Let $s=s(\gamma)$ come from Corollary~\ref{cor:main}, put $m'=R(m_0,2^{s^2})$, and take $k=k(\gamma,m')$ from that corollary.  Applied to $G$, it gives $s$-element multisets $S_1,\dots,S_{m'}$ in distinct color classes, with every pair density at least $d-\gamma$.

Label the occurrences of $S_i$ by $v_i^1,\dots,v_i^s$.  Color each pair $i<j$ by the matrix
$W_{ij}=\{(a,b)\in[s]^2:v_i^a v_j^b\in E(G)\}$.
Ramsey's theorem gives $m_0$ classes with a common matrix $W$.  The possibly looped digraph $D$ on $[s]$ with arc set $W$ satisfies $\dens(D)\ge d-\gamma$ and admits a colored homomorphism $D[m_0]\to G$.

It remains to find a $t$-vertex digraph $R$ with $R\to D$ and $\dens(R)\ge d-\varepsilon$.  If $D$ has a loop, take the complete looped digraph on $t$ vertices and map it to that loop.  Otherwise, for each ordered $t$-tuple $(X_1,\dots,X_t)$ of vertices of $D$, allow repetitions and define $R$ on $[t]$ by $ij\in E(R)$ if and only if $X_iX_j\in E(D)$.  Every such $R$ is loopless and maps to $D$.  Let $a$ be the maximum number of arcs among these digraphs.  Uniform independent sampling gives
$  a\ge t(t-1)\dens(D).$
Suppose that $a/t^2<d-\varepsilon$.  Since $d\le1$ and $\varepsilon\ge1/t$, integrality gives $a\le t(t-1)-1$.  Consequently,
\[
  d-\dens(D)
  >\frac{a}{t^2}+\frac1t-\frac{a}{t(t-1)}
  =\frac1t-\frac{a}{t^2(t-1)}
  \ge\frac{1}{t^2(t-1)},
\]
contrary to $\dens(D)\ge d-\gamma$.  A maximizing $R$ therefore has the required density.  Composing $R[m]\to D[m]\to G$ completes the proof.
\end{proof}

\begin{remark}[Sharpness of the template order]\label{rem:sharp-template}
Fix $0<\varepsilon<1$ and put $q=\lceil\varepsilon^{-1}\rceil-1$, so that $\frac1q>\varepsilon$.  Choose an integer $N$ with $\frac1N<\frac1q-\varepsilon$, and let $m=N+1$. We write $\overleftrightarrow K_N$ for the loopless complete bidirected digraph on $N$ vertices.  For every $k\ge m$, the graph $G=\overleftrightarrow K_N[k]$ is $d$-dense with $d=1-\frac1N$, and its underlying graph is $N$-partite.
If $R[m]\to G$ coloredly, then $R$ cannot have a loop, since its $m$ copies would form a $K_{N+1}$.  A loopless $R$ on at most $q$ vertices has density at most $1-\frac1q<d-\varepsilon$.  Thus no smaller bound on the template order works uniformly in $m$.
\end{remark}

\begin{proof}[Proof of Theorem~\ref{thm:digraph}]
An edgeless member gives value zero on both sides, so assume that every member of $\Fcal$ has an edge.  Put $\Hcal=\Dcal(\Fcal)$ and $\alpha=\pi(\Hcal)$.  If a finite digraph $D$ receives no homomorphism from any member of $\Hcal$, then Lemma~\ref{lem:half-blowup-test} shows that $D[k]$ receives no colored homomorphism from any member of $\Fcal$, for every sufficiently large $k$.  Hence $\pc(\Fcal)\ge\dens(D)$, and taking the supremum gives $\pc(\Fcal)\ge\alpha$.

Conversely, fix $\varepsilon>0$ and let $m$ be the maximum number of colors in a member of $\Fcal$.  Theorem~\ref{thm:blowup} gives $k$ such that every $(\alpha+2\varepsilon)$-dense $k$-colored graph $G$ receives a colored homomorphism from $D[m]$, where $\dens(D)\ge\alpha+\varepsilon$.  By the definition of $\alpha$, some $F^{\preceq}\in\Hcal$ maps to $D$.  Lemma~\ref{lem:half-blowup-test} then gives $F\to D[m]\to G$.  Thus $\pc(\Fcal)\le\alpha+2\varepsilon$; when $\alpha+2\varepsilon>1$ this bound is automatic.  Letting $\varepsilon\downarrow0$ proves the other inequality.
\end{proof}

\subsection{Finite forbidden families and exact realization}

The finite-support proposition gives one inclusion in the value-set theorem.  We prove the reverse inclusion by a direct finite construction.

\begin{theorem}\label{thm:exact-realization}
For every non-empty finite loopless digraph $D_0$, there exists a non-empty finite colored graph family $\Fcal$ such that $\pc(\Fcal)=\lambda(D_0)$.
\end{theorem}

\begin{proof}
Put $n=|V(D_0)|$ and $\gamma=\lambda(D_0)$.  Let $\mathcal Q$ contain one representative of each isomorphism type of non-empty loopless semicomplete digraph $Q$ satisfying
$  |V(Q)|<2^n$
and 
$  \lambda(Q)>\gamma.$
This is a finite family.  For every $Q\in\mathcal Q$, Lemma~\ref{lem:lag-basic} gives $Q\nrightarrow D_0$ and $Q^{\opp}\nrightarrow D_0$.  Choose $m_Q=n^{|V(Q)|}+1$ as in Lemma~\ref{lem:layered}, and define
$  \Fcal=\{K_{n+1}^{\mathrm{rb}}\}
  \cup\{Q[m_Q]:Q\in\mathcal Q\},$
where $K_{n+1}^{\mathrm{rb}}$ denotes the rainbow clique.

No orientation arising from a member of $\Fcal$ maps to $D_0$.  This follows from Lemma~\ref{lem:layered} for the half-blowups; for the rainbow clique, every orientation is a transitive tournament on $n+1$ vertices, which cannot map to a loopless $n$-vertex digraph.  By Lemma~\ref{lem:lag-rational}, some finite blowup $B\to D_0$ has density exactly $\gamma$.  It also avoids every member of $\Dcal(\Fcal)$, so Theorem~\ref{thm:digraph} gives $\pc(\Fcal)\ge\gamma$.

For the reverse inequality, Proposition~\ref{prop:acyclic-lagrangian} gives a loopless semicomplete $\Dcal(\Fcal)$-homomorphism-free digraph $R$ with $\lambda(R)=\pc(\Fcal)$.  Since a rainbow $K_{n+1}$ is forbidden, the proposition gives $|V(R)|<2^n$.  If $\lambda(R)>\gamma$, then $R$ is isomorphic to a member $Q\in\mathcal Q$.  But orienting $Q[m_Q]$ in the natural order of its layers and projecting onto $Q$ gives
$ (Q[m_Q])^{<}\to Q\cong R$,
contrary to the choice of $R$.  Hence $\pc(\Fcal)=\lambda(R)\le\gamma$, proving equality.
\end{proof}

\begin{proof}[Proof of Theorem~\ref{thm:value-set}]
By Theorem~\ref{thm:digraph}, every finite colored Tur\'an density is the homomorphism Tur\'an density of a finite acyclic digraph family.  Proposition~\ref{prop:acyclic-lagrangian} therefore gives $\Picol\subseteq\Ldig$.  Theorem~\ref{thm:exact-realization} gives the reverse inclusion.
\end{proof}

\begin{proof}[Proof of Corollary~\ref{cor:rational}]
Theorem~\ref{thm:value-set} and Lemma~\ref{lem:lag-rational} give rationality.  More precisely, when every forbidden member has an edge, Proposition~\ref{prop:acyclic-lagrangian} gives a $\Dcal(\Fcal)$-homomorphism-free template $Q$ attaining the value.  An optimal rational blowup $B\to Q$ remains homomorphism-free and has density exactly $\lambda(Q)=\pc(\Fcal)$.  By Lemma~\ref{lem:half-blowup-test}, every sufficiently large half-blowup $B[k]$ avoids $\Fcal$ and has that density.  The template $Q$ has fewer than $2^{h-1}$ vertices, where $h$ is the smallest order of a forbidden member.
\end{proof}

\begin{proof}[Proof of Property~\ref{property:no-right}]
By Theorem~\ref{thm:digraph} and Lemma~\ref{lem:hom-sub}, $\Picol$ is a subset of the well-ordered set in Theorem~\ref{thm:BES}.  Fix $\alpha<1$.  If $\Picol\cap(\alpha,1)$ is empty, take any $0<\varepsilon<1-\alpha$.  Otherwise this set has a least element $\beta>\alpha$, and $\varepsilon=(\beta-\alpha)/2$ has the required property.
\end{proof}

\subsection{Forbidding a single colored graph}

\begin{proof}[Proof of Corollary~\ref{cor:jump}]
Suppose first that $F$ does not admit a colored homomorphism into the half-graph on its color set.  The same is then true for every larger half-graph, because restricting to the image layers recovers a half-graph on the colors of $F$.  Since half-graphs are $\frac{1}{4}$-dense, $\pc(F)\ge\frac{1}{4}$.

Conversely, suppose that $F$ admits a colored homomorphism into a half-graph.  The order of the image layers induces an ordering $\preceq$ of the colors of $F$.  In $F^{\preceq}$, every vertex mapped to an $a$-vertex of the half-graph is a source and every vertex mapped to a $b$-vertex is a sink.  Hence $F^{\preceq}$ maps to the one-arc digraph.  Every digraph with at least one arc receives a homomorphism from this one-arc digraph, so Theorem~\ref{thm:digraph} gives $\pc(F)=0$.

To attain $\frac{1}{4}$, let $F$ be the path $x_1x_2x_3x_4x_5$ with color sequence $1-2-3-1-2$.  Under any ordering of the three colors, exactly one color lies between the other two.  The internal vertex of that color---$x_2$, $x_3$, or $x_4$---then has both an in-neighbor and an out-neighbor in $F^{\preceq}$.  Thus no $F^{\preceq}$ maps to the one-arc digraph, and $F$ maps to no half-graph.  Therefore $\pc(F)\ge\frac{1}{4}$.

For the reverse inequality, use the order $1\prec2\prec3$.  The resulting orientation $F^{\preceq}$ maps to the directed path $\vec P_3$.  Let $D$ be loopless and have density greater than $\frac{1}{4}$.  If every vertex of $D$ were a source or a sink, all arcs would go from a set of sources to a set of sinks, giving at most $\frac{|V(D)|^2}{4}$ arcs.  Hence $D$ contains a vertex with both positive in-degree and positive out-degree, and therefore receives a homomorphism from $\vec P_3$.  A looped host receives a homomorphism from $F^{\preceq}$ by collapsing all vertices to a loop.  Thus $\pi(\Dcal(F))\le\frac{1}{4}$, and Theorem~\ref{thm:digraph} yields $\pc(F)=\frac{1}{4}$.
\end{proof}

\begin{proof}[Proof of Corollary~\ref{cor:chromatic}]
Let $\ell=\chi(F)\ge2$ and let $\overleftrightarrow K_t$ denote the loopless complete bidirected digraph on $t$ vertices.  No orientation $F^{\preceq}$ maps to $\overleftrightarrow K_{\ell-1}$, since such a homomorphism would be a proper $(\ell-1)$-coloring of the underlying graph of $F$.  Therefore
\(
  \pc(F)=\pi(\Dcal(F))
  \ge \dens(\overleftrightarrow K_{\ell-1})
  =1-\frac{1}{\ell-1}.
\)

Set
\(
  \alpha=1-\frac{1}{2(\ell-1)}.
\)
Every looped digraph receives a homomorphism from $F^{\preceq}$.  If a loopless $n$-vertex digraph $D$ has density at least $\alpha$, and $b(D)$ denotes its number of bidirected unordered pairs, then
\[
  b(D)\ge e(D)-\binom n2
  \ge \alpha n^2-\binom n2
  >\left(1-\frac{1}{\ell-1}\right)\frac{n^2}{2}.
\]
Tur\'an's theorem therefore gives a bidirected clique of order $\ell$, and every $F^{\preceq}$ maps to this clique.

This excludes every $\Dcal(F)$-free digraph of density at least $\alpha$.  To obtain the strict inequality for the supremum, apply Proposition~\ref{prop:acyclic-lagrangian}: there is a finite loopless digraph $Q$ with
\(
  \pi(\Dcal(F))=\lambda(Q).
\)
By Lemma~\ref{lem:lag-rational}, $\lambda(Q)$ is the exact density of a finite blowup $B$ of a subgraph of $Q$.  The digraph $B$ is still $\Dcal(F)$-free.  If $\lambda(Q)\ge\alpha$, the preceding paragraph applied to $B$ gives a contradiction.  Hence $\pc(F)<\alpha$.

For sharpness, let $t$ be positive and let $F_{\ell,m}=K_\ell\times K_m$ have vertex set $[\ell]\times[m]$, with $(i,j)$ adjacent to $(i',j')$ exactly when $i\ne i'$ and $j\ne j'$.  Give $(i,j)$ color $j$.  For $m\ge\ell$, the underlying graph has chromatic number $\ell$.  Let $D_{\ell-1,t}$ have vertex set $[\ell-1]\times[t]$, with an arc from $(a,b)$ to $(a',b')$ whenever either $a\ne a'$ or $a=a'$ and $b<b'$.  Its density is
\(
  \dens(D_{\ell-1,t})
  =1-\frac{1}{2(\ell-1)}-\frac{1}{2(\ell-1)t}.
\)

Take
\(
  m>((\ell-1)t)^\ell.
\)
All members of $\Dcal(F_{\ell,m})$ are isomorphic, so order the colors naturally.  Suppose that $f:F_{\ell,m}^{\preceq}\to D_{\ell-1,t}$ were a homomorphism.  There are only $((\ell-1)t)^\ell$ possible image profiles of a color class, so two colors $j<j'$ have the same profile:
\(
f(i,j)=f(i,j')\) for every 
\(i\in[\ell].
\)
Among the $\ell$ images $f(i,j)$, two, say
\(
  f(i,j)=(a,b)\), and \( f(i',j)=(a,b'),
\)
have the same first coordinate.  Relabel $i,i'$ so that $b\ge b'$.  The arc $(i,j)\to(i',j')$ must then map to $(a,b)\to(a,b')$, which is not an arc of $D_{\ell-1,t}$.  This contradiction shows that $D_{\ell-1,t}$ is $\Dcal(F_{\ell,m})$-free.  Hence
\[
  \pc(F_{\ell,m})
  \ge1-\frac{1}{2(\ell-1)}-\frac{1}{2(\ell-1)t}.
\]
Letting $t\to\infty$ proves the accumulation statement.
\end{proof}

\subsection{An explicit non-principal pair}

\begin{proof}[Proof of Corollary~\ref{cor:nonprincipal}]
Let $F_1=K_2\times K_{17}$ have vertices $a_i,b_i$ of color $i$, for $i\in[17]$, and edges $a_i b_j$ whenever $i\ne j$.  Let $F_2$ be a rainbow triangle.  We show that
\[
  \pc(F_1,F_2)=\frac13
  <\frac38\le\pc(F_1),
  \qquad \pc(F_2)=\frac12.
\]
Every permutation of the color classes is an automorphism of $F_1$, so all members of $\Dcal(F_1)$ are isomorphic to the orientation given by the natural color order.

Suppose this orientation maps to a loopless digraph $D$ with at most four vertices and no pair of opposite arcs.  There are at most $|V(D)|^2\le16$ possible ordered pairs $(f(a_i),f(b_i))$, so two layers $i<j$ have the same pair $(u,v)$.  The arcs $a_i b_j$ and $b_i a_j$ then force both $uv$ and $vu$ in $D$, a contradiction.  Taking $D=T_4$ gives $\pc(F_1)\ge\dens(T_4)=3/8$.  Taking $D=\vec C_3$, which also avoids $T_3$, gives $\pc(F_1,F_2)\ge1/3$ by Theorem~\ref{thm:digraph}.

For the upper bound, let $D$ receive no homomorphism from $\Dcal(F_1)\cup\Dcal(F_2)$.  It has no loop.  Since $F_1$ is bipartite, every orientation of $F_1$ maps to a bidirected pair, so $D$ has no such pair either.  Also, $\Dcal(F_2)=\{T_3\}$ up to isomorphism.  The underlying graph of $D$ is therefore $K_4$-free: a clique on four vertices would give a tournament, and every tournament on four vertices contains $T_3$.  Tur\'an's theorem yields $e(D)\le |V(D)|^2/3$, because each underlying edge contributes exactly one arc.  Hence $\pc(F_1,F_2)\le1/3$.  Finally, $\pc(F_2)=1/2$ is the usual rainbow triangle density.
\end{proof}

\section{Proof of the minimum-degree results}\label{sec:min-degree-proofs}

The minimum-degree reduction follows from the stronger sparsification theorem and the same Ramsey step.  No further compression of the directed template is needed.

\begin{proof}[Proof of Theorem~\ref{thm:blowdeg}]
Let $s=s(\varepsilon)$ and $n=n(\varepsilon)$ be supplied by Theorem~\ref{thm:technical}.  Put $m_0=\max\{m,2\}$, set
\(
  m'=R(m_0,2^{s^2}),
\)
and take $k=k(\varepsilon,m')$ from that theorem.

Let $G$ be a $k$-colored graph of degree $\delta$.  If a color class has fewer than $n$ vertices, replace every vertex in that class by the same number of twins so that the enlarged class has size at least $n$.  This preserves every relative degree into that class, and the enlarged graph projects by a colored homomorphism to $G$.  We may therefore assume that all classes have size at least $n$.

Theorem~\ref{thm:technical} gives $m'$ classes and $s$-subsets $S_1,\dots,S_{m'}$ such that each pair is an $\varepsilon$-discretization of the corresponding original pair.  Hence every vertex in $S_i$ has at least $(\delta-\varepsilon)s$ neighbors in $S_j$.  Label the vertices of $S_i$ by $v_i^1,\dots,v_i^s$ and, for $i<j$, put
\[
  W_{ij}=\{(a,b)\in[s]^2:v_i^a v_j^b\in E(G)\}.
\]
Ramsey's theorem gives $m_0$ indices on which all sets $W_{ij}$ equal a common set $W$.  Let $D$ be the possibly looped digraph on $[s]$ with arc set $W$.  For any selected pair $i<j$, the row degrees of $W_{ij}$ are the out-degrees in $D$, and its column degrees are the in-degrees in $D$.  Since $m_0\ge2$, such a pair exists, and the minimum-degree condition between $S_i$ and $S_j$ shows that the minimum out-degree and minimum in-degree of $D$ are both at least $(\delta-\varepsilon)s$.  The selected classes realize a colored homomorphism $D[m_0]\to G$; restricting it to any $m$ selected layers gives $D[m]\to G$.  Composing with the projection from the enlarged host to the original one, when necessary, proves the theorem.
\end{proof}

\begin{proof}[Proof of Theorem~\ref{thm:degdigraph}]
If some $F\in\Fcal$ is edgeless, then every host with at least as many colors as $F$ receives a colored homomorphism from $F$, while the edgeless member of $\Dcal(F)$ maps to every non-empty digraph.  By the zero conventions in the definition, both sides are therefore zero.  Assume henceforth that every member of $\Fcal$ has an edge, and put
\(
  d=\pi^\delta(\Dcal(\Fcal)).
\)
For $\eta>0$, choose a digraph $H$ with minimum in-degree and out-degree at least $(d-\eta)|V(H)|$ that admits no homomorphism from any member of $\Dcal(\Fcal)$.  For every positive integer $k$, the half-blowup $H[k]$ has degree at least $d-\eta$; members of $\Fcal$ with more than $k$ colors cannot map to it, and Lemma~\ref{lem:half-blowup-test} excludes the remaining members.  Thus it admits no colored homomorphism from any $F\in\Fcal$.  Letting $\eta\downarrow0$ gives $\pcd(\Fcal)\ge d$.

For the reverse inequality, suppose that $\pcd(\Fcal)>d$ and choose $\eta>0$ with $d+2\eta<\pcd(\Fcal)$.  Let $m$ be at least the maximum number of colors in a member of $\Fcal$, and take $k=k(\eta,m)$ from Theorem~\ref{thm:blowdeg}.  By the definition of $\pcd$, there is a $k$-colored graph $G$ of degree at least $d+2\eta$ that avoids every member of $\Fcal$.  Theorem~\ref{thm:blowdeg} gives a digraph $D$ with minimum in-degree and out-degree at least $(d+\eta)|V(D)|$ and a colored homomorphism $D[m]\to G$.

If some $F^{\preceq}\in\Dcal(\Fcal)$ mapped to $D$, Lemma~\ref{lem:half-blowup-test} would give
\(
  F\to D[m]\to G,
\)
a contradiction.  Thus $D$ is $\Dcal(\Fcal)$-free with minimum semidegree greater than $d|V(D)|$, contradicting the definition of $d$.  Therefore $\pcd(\Fcal)\le d$.
\end{proof}

To prove Theorem~\ref{thm:nojump}, we use a directed-girth bound. The Caccetta--H\"aggkvist conjecture states that, for positive integers $n$ and $r$, every digraph on $n$ vertices with minimum out-degree at least $r$ contains a directed cycle of length at most $\left\lceil\frac{n}{r}\right\rceil$. We will use an upper bound on the cycle length proved by Shen \cite{She02}:

\begin{theorem}[\cite{She02}]\label{thm:shencycles}
For every integer $k\ge74$, every oriented graph on $n$ vertices with minimum out-degree at least $n/(k-73)$ contains a directed cycle of length at most $k$.
\end{theorem}

\begin{proof}[Proof of Theorem~\ref{thm:nojump}]
For $k\ge2$, let $F_k$ be the $3$-colored cycle $C_{3k!}$ with periodic color sequence $1-2-3$.  Along the cyclic order of the underlying cycle, the six linear orders of the colors produce the three cyclic shifts of the periodic sign pattern $++-$ and the three cyclic shifts of $--+$.  Rotations account for the cyclic shifts, and a reflection identifies the two types.  Thus all six resulting digraphs are isomorphic; denote their common isomorphism type by $D_k$.  Its vertices are $v_1,\dots,v_{3k!}$, with indices read cyclically, and the edge $v_i v_{i+1}$ is oriented toward $v_i$ when $3\mid i$ and toward $v_{i+1}$ otherwise.  Consequently, Theorem~\ref{thm:degdigraph} gives
\(\pcd(F_k)=\pi^\delta(\Dcal(F_k))=\pi^\delta(D_k),
\)
so it suffices to prove that $0<\pi^\delta(D_k)\to0$.

Let $\vec C_\ell$ have cyclically indexed vertices $w_0,\dots,w_{\ell-1}$.  Once the image of $v_1$ is fixed, the arcs of $D_k$ force the images of all subsequent vertices.  More precisely, if $f(v_i)=w_j$, then every block of three consecutive edges gives
\(
  f(v_{i+3})=w_{j+1},
\)
where the second subscript is read modulo $\ell$.
After the $k!$ blocks have been traversed, the map closes if and only if $k!\equiv0\pmod\ell$.  Conversely, when $\ell\mid k!$, the same recurrence defines a homomorphism.  Thus
\(
  D_k\to \vec C_\ell\) if and only if
  \(\ell\mid k!.
\)
Choose a prime $p>k$.  Then $D_k\nrightarrow\vec C_p$, while $\vec C_p$ has minimum in-degree and out-degree one.  Hence
\(
  \pi^\delta(D_k)\ge\frac{1}{p}>0.
\)

For the upper bound, let $k\ge74$ and suppose that an $n$-vertex digraph $H$ is $D_k$-free with minimum in-degree and out-degree at least $\frac{n}{k-73}$.  Such a host has no loop, since a loop receives a homomorphism from $D_k$, and it has no pair of opposite arcs, since $2\mid k!$ and $D_k\to\vec C_2$.  Thus $H$ is an oriented graph, and Theorem~\ref{thm:shencycles} gives a directed cycle $\vec C_\ell$ with $\ell\le k$.  Since $\ell\mid k!$, we have $D_k\to\vec C_\ell\to H$, a contradiction.  Therefore
\(
  \pi^\delta(D_k)\le\frac{1}{k-73}.
\)
Combining the two bounds gives
\[
  0<\pcd(F_k)=\pi^\delta(D_k)\le\frac{1}{k-73}.
\]
Given $\varepsilon>0$, choosing $k$ so large that $\frac{1}{k-73}<\varepsilon$ proves the theorem.
\end{proof}

\section{Concluding remarks}\label{sec:remarks}

The directed reduction determines the finite-family density values, while the low-resolution theorem supplies a sampling tool for arbitrary prescribed partitions.  Three questions remain particularly natural.

\medskip

\noindent\emph{A single forbidden colored graph.}
Theorem~\ref{thm:value-set} realizes every finite loopless digraph Lagrangian by a finite colored family.  Does every such value already equal $\pc(F)$ for a single colored graph $F$?  Equivalently, does restricting to singleton families change the value set?  Non-principality does not settle this question: it compares a family with its own members, rather than with all possible single obstructions.  A characterization would complement both the realization theorem and the chromatic-number bounds in Corollary~\ref{cor:chromatic}.

\medskip

\noindent\emph{The quantitative cost of sparsification.}
Theorem~\ref{thm:technical} gives sample size $O(\varepsilon^{-3}\log(\frac2\varepsilon))$, independent of the number of retained classes.  Determine the optimal sample size, both for density preservation alone and for the stronger degree-profile and pointwise-degree conclusion.  It is also important to understand the necessary growth of $k(\varepsilon,m)$: the proof uses repeated hypergraph Ramsey arguments, whereas the final samples are only polynomial in $\frac1\varepsilon$.  Sharp lower bounds separating the cost of selecting the classes from that of sampling within them would clarify the scope of the method.

\medskip

\noindent\emph{Minimum-degree values and directed girth.}
Theorem~\ref{thm:nojump} rules out a positive gap above zero, but it does not give a corresponding description of the minimum-degree value set.  In particular, must $\pcd(\Fcal)$ be rational for every non-empty finite colored family?  A concrete test case is the periodically colored cycle $F_k=C_{3k!}$ used in the proof: its associated orientation maps to $\vec C_\ell$ exactly when $\ell\mid k!$.  Determining $\pcd(F_k)$ would identify how much of this extremal problem is governed by directed-cycle constructions, and how much requires more general digraphs.  Unlike the edge-density case, the Lagrangian support reduction does not preserve minimum semidegree, so a different structural argument is needed.

\bibliographystyle{abbrv}
\bibliography{bibliog}

\end{document}